\documentclass[11pt,a4paper]{article}
\usepackage[a4paper,hmargin=2.8cm,vmargin=3cm]{geometry}
\usepackage[utf8x]{inputenc}
\usepackage{amsmath,amsthm,amsfonts,amssymb}
\usepackage{authblk}
\usepackage[bookmarksnumbered=true]{hyperref}
\usepackage{bbm}
\usepackage{bigints}
\usepackage{esint}
\usepackage{stix}

\usepackage{color}
\usepackage{appendix}

\usepackage{graphicx}

\numberwithin{equation}{section}

\newtheorem{theorem}{Theorem}[section]
\newtheorem{corollary}[theorem]{Corollary}
\newtheorem{lemma}[theorem]{Lemma}

\theoremstyle{definition}
\newtheorem{definition}[theorem]{Definition}
\newtheorem{remark}[theorem]{Remark}

\let\oldmarginpar\marginpar
\renewcommand\marginpar[1]{\-\oldmarginpar[\raggedleft\footnotesize #1]%
  {\raggedright\footnotesize #1}}

\def\eps{\varepsilon}

\def\R{\mathbb{R}}

\let\.=\cdot
\let\0=\emptyset

\def\loc{\text{\rm loc}}

\def\square{\hbox{$\sqcap\kern-7pt\sqcup$}}

\def\be{\begin{equation}}
\def\ee{\end{equation}}
\def\bea{\begin{eqnarray}}
\def\eea{\end{eqnarray}}

\def\ˆ{^{•}}

 \newcommand{\dive}{{\rm{div}}}

\newcommand{\La}{\mathcal{L}}
 
\newcommand{\varp}{\varphi}  
 
\newcommand{\om}{\Omega} 
\newcommand{\p}{\partial} 
 
\newcommand{\la}{\lambda} 
 
\newcommand{\IR}{\mathbb{R}}
\newcommand{\irn}{\IR^N}

\newcommand{\ID}{\mathcal{D}}

\title{A note on the Lagrangian flow \\ associated to a partially regular vector field}

\author{Gianluca Crippa, Silvia Ligabue}

\def\adresse{
\begin{description}
\item[G. Crippa:] 
Departement Mathematik und Informatik,\\ Universit\"{a}t Basel,
Spiegelgasse 1,
CH-4051 Basel, Switzerland\\
E-mail: \texttt{gianluca.crippa@unibas.ch }

\item[S. Ligabue:] Departement Mathematik und Informatik,\\ Universit\"{a}t Basel,
Spiegelgasse 1,
CH-4051 Basel, Switzerland\\
E-mail: \texttt{silvia.ligabue@unibas.ch}

\end{description}
}

\date{\today}

\begin{document}

\maketitle

\begin{abstract}
\noindent In this paper we derive quantitative estimates for the Lagrangian flow associated to a partially regular vector field of the form
$$
b(t,x_1,x_2) = (b_1(t,x_1),b_2(t,x_1,x_2)) \in\R^{n_1}\times\R^{n_2} \,, \qquad (x_1,x_2)\in\R^{n_1}\times\R^{n_2}\,.
$$
We assume that the first component $b_1$ does not depend on the second variable $x_2$, and has Sobolev~$W^{1,p}$ regularity in the variable $x_1$, for some $p>1$. On the other hand, the second component $b_2$ has Sobolev  $W^{1,p}$ regularity in the variable $x_2$, but only fractional Sobolev $W^{\alpha,1}$ regularity in the variable $x_1$, for some $\alpha>1/2$. These estimates imply well-posedness, compactness, and quantitative stability for the Lagrangian flow associated to such a vector field. 
\end{abstract}

\section{Introduction}

The transport equation
\begin{equation}\label{e:transport}
\partial_t u + b \cdot \nabla u=0
\end{equation}
is one of the basic building blocks for several (often nonlinear) partial differential equations (PDEs) from mathematical physics, most notably from fluid dynamics, conservation laws, and kinetic theory. In~\eqref{e:transport} the vector field $b = b(t,x) : (0,T) \times \R^N \to \R^N$ is assumed to be given, hence~\eqref{e:transport} is a linear equation for the unknown $u = u(t,x) : (0,T) \times \R^N \to \R$, with a prescribed initial datum~$u(0,x) = \bar u (x)$. Physically, the solution $u$ is advected by the vector field $b$. In most applications~\eqref{e:transport} is coupled to other PDEs, and moreover the vector field is often not prescribed, but rather depends on the other physical quantities present in the problem. Nevertheless, a thorough understanding of the linear equation~\eqref{e:transport} is often the basic step for the treatment of such nonlinear cases.

If the vector field is regular enough (Lipschitz in the space variable, uniformly with respect to time) the well-posedness of~\eqref{e:transport} is classically well-understood and is based on the theory of characteristics and on the connection with the ordinary differential equation (ODE)
\begin{equation}\label{e:ode}
\left\{\begin{array}{l}
\displaystyle\frac{d}{dt} X(s,x) = b(s,X(s,x)) \\ \\ 
X(0,x) = x \,.
\end{array}\right.
\end{equation}
The map $X = X(t,x) : (0,T) \times \R^N \to \R^N$ is the (classical) flow associated to the vector field $b$.

When dealing with problems originating from mathematical physics, however, the regularity available on the advecting vector field is often much lower than Lipschitz, and this prevents the application of the classical theory. The low regularity of the vector field usually accounts for ``chaotic'' and ``turbulent'' behaviours of the system. This is the reason why in the last few decades a systematic study of~\eqref{e:transport} and~\eqref{e:ode} out of the Lipschitz regularity setting has been carried out. We mention in particular the seminal papers by DiPerna and Lions~\cite{DPL} and Ambrosio~\cite{Amb}, where respectively Sobolev and bounded variation regularity have been assumed on the vector field, together with assumptions of boundedness of the (distributional) spatial divergence and on the growth of the vector field. We will now (briefly and informally) describe the main points of the theory, and we refer for instance to the survey article~\cite{heriott} for more details.

The approach in~\cite{DPL,Amb} is based on the notion of renormalized solution of~\eqref{e:transport}. Formally at least, a strategy to prove uniqueness for~\eqref{e:transport} consists in deriving energy estimates: multiplying~\eqref{e:transport} by~$2u$, integrating in space, and integrating by parts, one obtains
\begin{equation}\label{e:energy}
\frac{d}{dt} \int_{\R^N} u(t,x)^2 \, dx \leq \| \dive b\|_\infty \int_{\R^N} u(t,x)^2 \, dx \,.
\end{equation} 
If the divergence of the vector field is bounded, Gr\"onwall lemma together with the linearity of~\eqref{e:transport} implies uniqueness. However, the formal computations leading to~\eqref{e:energy} cannot be made rigorous without any regularity assumptions: when dealing with weak solutions of~\eqref{e:transport}, which do not enjoy any regularity beyond integrability, it is not justified to apply the chain rule in order to get the identities
$$
2u\partial_t u = \partial_t u^2
\qquad \text{ and } \qquad 
2u\nabla u = \nabla u^2 \,.
$$
Following~\cite{DPL}, we say that a bounded weak solution $u$ of~\eqref{e:transport} is a renormalized solution if
\begin{equation}\label{e:renorm}
\partial_t \beta(u) + b \cdot \nabla \beta(u) = 0
\end{equation}
holds in the sense of distributions for every smooth function $\beta : \R\to\R$. Roughly speaking, renormalized solutions are the class inside which the energy estimate~\eqref{e:energy} can be made rigorous. The problem is then switched to proving that all weak solutions are renormalized. To achieve this, one can regularize~\eqref{e:transport} by convolving with a regularization kernel $\rho_\eps(x)$, obtaining 
$$
\partial_t u_\eps + b \cdot \nabla u_\eps = b \cdot \nabla u_\eps - (b \cdot \nabla u)\ast\rho_\eps =: R_\eps\,,
$$
where we denote $u_\eps = u \ast \rho_\eps$ and the right hand side $R_\eps$ is called commutator. Multiplying this equation by $\beta'(u_\eps)$ we obtain
\begin{equation}\label{e:aftmult}
\partial_t \beta(u_\eps) + b \cdot \nabla \beta(u_\eps) = R_\eps \, \beta'(u_\eps)\,,
\end{equation} 
which implies~\eqref{e:renorm} provided the commutator $R_\eps$ converges to zero strongly. Such a convergence holds under Sobolev regularity assumptions on the vector field $b$, as can be proved by rewriting the commutator as an integral involving difference quotients of the vector field. This strategy has been pursued in~\cite{DPL} to show uniqueness and stability of weak solutions of~\eqref{e:transport} in the case of Sobolev vector fields, and extended (with several nontrivial modifications) by Ambrosio~\cite{Amb} to the case of vector fields with bounded variation. The convergence to zero of the right hand side of~\eqref{e:aftmult} is more complex in this last setting, and the convolution kernel $\rho_\eps$ has to be properly chosen in a way which depends on the vector field itself. 

An alternative approach has been developed in~\cite{CDL}, working at the level of the ODE~\eqref{e:ode} and deriving a priori estimates for a functional measuring a ``logarithmic distance'' between two flows associated to the same vector field, namely
\begin{equation}\label{e:basfunc}
\Phi_\delta(s) = \int \log \left( 1 + \frac{|X(s,x)-\bar X(s,x)|}{\delta}\right) \, dx \,,
\end{equation}
where $\delta>0$ is a small parameter which is optimized in the course of the argument. Differentiating the functional $\Phi_\delta$ in time one can estimate
$$
\Phi'_\delta(s) \lesssim \int \frac{|b(s,X(s,x))-b(s,\bar X(s,x))|}{|X(s,x)-\bar X(s,x)|} \, dx
\lesssim \int \big[ MDb(s,X(s,x)) + MDb(s,\bar X(s,x))\big] \, dx \,,
$$
where in the second inequality we have estimated the difference quotients of $b$ with the maximal function of $Db$ (see Definition~\ref{d:maxfct} and Lemma~\ref{firstlemma}). Changing variable along the flows $X$ and $\bar X$ (which are assumed to have controlled compressibility), and recalling that the maximal function satisfies the so-called strong inequality $\|Mf\|_{L^p}\lesssim\|f\|_{L^p}$ when $1<p\leq\infty$ (see Lemma~\ref{l:maxest}), we find that $\Phi_\delta$ is uniformly bounded in $s$ and in $\delta$ if $b\in W^{1,p}$ with $1<p\leq\infty$. Together with the estimate 
$$
\mathcal{L}^{N}\big( \big\{ |X(s,x)-\bar X(s,x)|>\gamma \big\}\big) \leq \frac{\Phi_\delta(s)}{\log\left(1+ \frac{\gamma}{\delta}\right)} \qquad\qquad \forall \gamma>0\,,
$$
letting $\delta\to 0$ implies that $X=\bar X$ almost everywhere.

This and related estimates have been used in~\cite{CDL} and in several subsequent papers in order to prove uniqueness, compactness, stability, and mild regularity for the flow. The main advantage of this approach lies in its quantitative character. Let us mention that the same approach can also be used in some regularity settings not covered by the approach of~\cite{DPL,Amb}, as for instance in~\cite{BC,BBC,Hung}. 

We would like to remark that both approaches (renormalization and estimates for the ODE) require information on a full derivative of the vector field, even though in a suitable weak sense (Sobolev or $BV$ regularity, derivative which is a singular integral of an integrable function\ldots), with an integrable control with respect to time. This kind of assumption is in general sharp for the well-posedness, as shown by various counterexamples (\cite{Dep,CLR,Aiz,DPL,ABC1,ABC2}). However, under more special ``structural'' conditions on the vector field, well-posedness can be proved even for vector fields with ``less than one derivative'', see for instance~\cite{ABC1,ABC2} in the two-dimensional setting and~\cite{JC} for the Hamiltonian case in general dimension. 

A further case enjoying a ``special structure'' is that of partially regular vector fields as in~\cite{LBL,Ler1,Ler2}. Let us describe this case in some more detail. We assume to have a splitting of the space as $\R^N = \R^{n_1}\times\R^{n_2}$ and we denote the variable by~$x=(x_1,x_2)$. We consider a vector field of the form
\begin{equation}\label{e:partialfield}
b = (b_1,b_2)\,, \qquad \text{ with} \qquad 
b_1 = b_1(t,x_1)\,, \qquad b_2=b_2(t,x_1,x_2)\,,
\end{equation} 
where $b_1$ is assumed to be Sobolev (respectively, $BV$) in $x_1$, and $b_2$ is assumed to be Sobolev (respectively, $BV$) in $x_2$, but merely integrable in $x_1$, see~\cite{LBL,Ler1} (respectively,~\cite{Ler2}). Compared to the theory in~\cite{DPL,Amb}, no regularity is required for $b_2$ in the variable $x_1$; this is due to the strong requirement that $b_1$ does not depend on $x_2$. The authors in~\cite{LBL,Ler1,Ler2} address the PDE problem relying on the renormalization theory, with the additional idea to use two regularization kernels, namely $\rho_{\eps_1} = \rho_{\eps_1}(x_1)$ and $\rho_{\eps_2} = \rho_{\eps_2}(x_2)$, and to eventually send $\eps_1\to 0$ first, and then $\eps_2 \to 0$. Roughly speaking, this gives rise to commutators ``in $x_1$ only'' for $b_1$ and ``in $x_2$ only'' for $b_2$.

In this paper we exploit the Lagrangian approach from~\cite{DPL} in order to derive well-posedness and quantitative estimate for the flow associated to a vector field of the form~\eqref{e:partialfield}. As in~\cite{LBL,Ler1,Ler2} we exploit the anisotropy of the problem and we employ different scales in $x_1$ and $x_2$. However, this is not done by convolving the PDE with the two kernels $\rho_{\eps_1}(x_1)$ and $\rho_{\eps_2}(x_2)$, but rather relying on an anisotropic variant (introduced in~\cite{BBC}) of the Lagrangian functional~\eqref{e:basfunc}, namely
\begin{equation}\label{e:anfunc}
\Phi_{\delta_1,\delta_2}(s) = \int \log\left( 1 + \frac{|X_1-\bar X_1|}{\delta_1} + 
\frac{|X_2-\bar X_2|}{\delta_2}\right)\, dx\,,
\end{equation}
where $\delta_1 \leq \delta_2$ (see~\eqref{e:functional} below for the exact expression of the functional we will use). 

In fact, due to the structure of the proof, we cannot send the two parameters $\delta_1$ and $\delta_2$ to zero one after the other; they are however related, and $\delta_1$ will be taken to be much smaller than $\delta_2$. This will reflect in the need for some regularity on $b_2$ in the variable $x_1$; however, we will need only a derivative of fractional order (more specifically, higher that $1/2$, see assumption~{\bf (R2)} in 
Section~\ref{ss:assumpt} for the precise statement).

Let us explain the key steps in our argument. Directly differentiating~$\Phi_{\delta_1,\delta_2}$ in time and arguing as in~\cite{BBC} we get
$$
\Phi_{\delta_1,\delta_2}(s) \lesssim \| D_{x_1}b_1\| + \frac{\delta_1}{\delta_2} \| D_{x_1}b_2\| + \| D_{x_2}b_2\|\,,
$$
with suitable norms on the right hand side, which depend on which exact regularity we assume on the vector field. The ratio $\delta_1/\delta_2$ can indeed be taken very small, but since $b_2$ does not possess a full derivative with respect to $x_1$, the term $\| D_{x_1}b_2\|$ is not bounded.

We can fix this issue by regularizing $b_2$ in the variable $x_1$ at scale $\eps>0$. In this way we get:
\begin{equation}\label{e:finheur}
\Phi_{\delta_1,\delta_2}(s) \lesssim \frac{\| b_2^\eps-b_2\|}{\delta_1} + \| D_{x_1}b_1\| + \frac{\delta_1}{\delta_2} \| D_{x_1}b_2^\eps \| + \| D_{x_2}b_2\| \lesssim C + \frac{\eps^\alpha}{\delta_1}+\frac{\delta_1}{\delta_2}\eps^{\alpha-1}\,,
\end{equation} 
where in the second inequality we used that
$$
\| b_2^\eps-b_2\| \sim \eps^\alpha
\qquad\text{ and }\qquad
\| D_{x_1}b_2^\eps \| \sim \eps^{\alpha-1}\,,
$$
assuming that $b_2$ possesses a derivative of order $\alpha$ in $x_1$ (see Lemma~\ref{lemma_W^{s,p}}). Taking $\delta_1 = \delta_2 \eps^{1-\alpha}$ the right hand side of~\eqref{e:finheur} takes the form $C + \eps^{2\alpha-1}/\delta_2$, which can be made bounded as $\delta_2\to 0$ by a suitable choice of $\eps$ provided $\alpha>1/2$. This is the reason why, with this approach, we need some fractional regularity of $b_2$ in $x_1$. From this bound on $\Phi_{\delta_1,\delta_2}$ all results on the well-posedness and further properties of the flow follow as in~\cite{BC}, see Section~\ref{ss:coroll} for the precise statements. 

\subsection*{Acknowledgements} This work has been partially supported by the Swiss National Science Foundation grant~200020\_156112 and by the ERC Starting Grant 676675 FLIRT. The authors gratefully acknowledge useful discussions with A.~Bohun on a preliminary version of these results.


\section{Preliminaries}

\subsection{Regular Lagrangian flows}

In the context of non-Lipschitz vector fields, the right concept of solution of the ordinary differential equation~\eqref{e:ode} is that of regular Lagrangian flow (see~\cite{DPL,Amb,heriott}). In the following, we are going to assume that
the vector field $b : (0,T) \times \R^N \to \R^N$ satisfies the following growth condition:
\begin{equation}
\label{growth_on_b}
\begin{split}
\mathbf{(R1)}: \quad &\frac{b(s,x)}{1+|x|} = c_1(s,x) + c_2(s,x)\,, \\
&\text{with } \quad c_1 \in L^1((0,T);L^1(\R^N)) \quad\text{ and }\quad c_2 \in L^1((0,T);L^\infty(\R^N)) \,.
\end{split}
\end{equation}

\begin{definition}[Regular Lagrangian flow]\label{d:rlf}
Let $b$ be a vector field satisfying {\bf (R1)}. A map 
$$
X\in C([0,T];L^0_{\loc}(\R^N))\cap \mathcal{B}([0,T];\log L_{\loc}(\R^N))
$$ 
is a regular Lagrangian flow in the renormalized sense relative to $b$ if:
\begin{enumerate}
\item The equation 
$$\p_s\bigl(\beta(X(s,x))\bigl)=\beta ' (X(s,x))b(s,X(s,x))$$ holds in $\ID'((0,T)\times\irn)$, for every function $\beta\in C^1(\irn;\IR)$ that satisfies
$$
|\beta(z)|\leq C(1+\log(1+|z|)) \quad \text{ and }\quad 
|\beta'(z)|\leq\frac{C}{1+|z|} \quad \text{  for all $z\in\R^N$,}
$$
\item $X(0,x)=x$ for a.e $x\in\irn$, 
\item There exists a constant $L\geq 0$ such that $\int_{\irn} \varp(X(s,x))dx\leq L\int_{\irn} \varphi(x)dx$ for all continuous functions  $\varphi:\R^N\rightarrow[0,\infty)$.  The constant $L$ is called {\em compressibility constant} of the flow.
\end{enumerate}
\end{definition}
In the above definition, $L^0_{\loc}$ denotes the space of measurable functions endowed with the local convergence in measure, $\mathcal{B}$ denotes the space of bounded functions, and $\log L_{\loc}$ denotes the space of locally logarithmically integrable functions. 

Given a vector field satisfying {\bf (R1)}, we can estimate the measure of the superlevels of the associated regular Lagrangian flow thanks to the following lemma:
\begin{lemma}\label{estsuper}
Let $b:(0,T)\times\R^N \to \R^N$ be a vector field satisfying  {\bf (R1)} and let $X:[0,T]\times\R^N\rightarrow\R^N$ be a regular Lagrangian flow relative to $b$ with compressibility constant $L$. Define the sublevels of the flow as
\begin{equation}\label{e:sublevels}
G_\lambda=\{x\in\R^N: |X(s,x)|\leq\la \text{ for almost all } s\in[0,T]\} \,.
\end{equation}
Then for all $r,\lambda>0$ it holds
$$
\La^N(B_r\setminus G_\lambda)\leq g(r,\lambda)\,,
$$
where the function $g$ depends only on $L$, $\|c_1\|_{L^1((0,T);L^1(\irn))}$ and $\|c_2\|_{L^1((0,T);L^\infty(\irn))}$ and satisfies $g(r,\lambda)\downarrow 0$ for $r$ fixed and $\lambda\uparrow \infty$. 
\end{lemma}

\subsection{Fractional Sobolev spaces}

We will make use of fractional Sobolev spaces according to the Sobolev--Slobodeckij definition:
\begin{definition}[Fractional Sobolev--Slobodeckij space] Let $f : \R^n \to \R$ be an integrable function, $f \in L^1(\R^n)$. Given $0<s<1$ and $1\leq p <\infty$, we say that $f \in W^{s,p}(\R^n)$ if 
$$
\int_{\R^n} \int_{\R^n} \frac{ |f(x)-f(y)|^{p}}{|x-y|^{sp+n}} \, dy \, dx  < \infty \,.
$$
\end{definition} 

The following lemma gives a rate of convergence of the convolution to the original function, and a rate of blow-up of the derivative of the function, under the assumption of fractional Sobolev regularity. 

\begin{lemma}
\label{lemma_W^{s,p}}
Let $f \in W^{s,p}(\R^n)$ and let $f^{\eps}$ be the convolution of $f$ with the standard mollifier $\varphi^{\eps}$. Then we have
\begin{equation}
\|f-f^{\eps}\|_{L^{p}(\R^{n})}\leq C  \eps^{s}\|f\|_{W^{s,p}(\R^n)} \qquad \text{ and } \qquad 
\|Df^{\eps}\|_{L^{p}(\R^{n})} \leq C  \eps^{s-1}\|f\|_{W^{s,p}(\R^n)}.
\end{equation}

\end{lemma}
\begin{proof}
For the first estimate we compute 
\begin{align*}
&\|f-f^{\eps}\|^p_{L^{p}(\R^{n})} = \int_{\R^n}|f(x)-f^{\eps}(x)|^{p} dx = \int_{\R^n}\left|f(x)-\int_{\R^n} \!\!f(x-y)\varphi_\eps (y)\, dy \right|^{p} dx \\
& = \int_{\R^n}\left|f(x)-\int_{\R^n}\! \!f(x-y)\varphi\left(\frac{y}{\eps}\right)\frac{1}{\eps^{n}}\, dy \right|^{p} dx= \int_{\R^n}\left|f(x)-\int_{\R^n} \!\!f(x-\eps z)\varphi(z)\frac{1}{\eps^{n}}\,\eps^{n} dz \right|^{p} dx \\
& =\int_{\R^n}\left|f(x)-\int_{\R^n} \!\!f(x-\eps z)\varphi(z)\, dz \right|^{p} dx = \int_{\R^n} \left|\int_{\R^n} [f(x)\varphi(z)-f(x-\eps z)\varphi(z)] dz \right|^{p}  dx\\
& =\int_{\R^n} \left| \int[f(x)-f(x-\eps z)]\varphi(z) dz \right|^{p} dx \leq \int_{\R^n}\int_{\R^n} |f(x)-f(x-\eps z)|^{p}\varphi(z)\, dz\, dx   \\
& \leq \int_{\R^n} \int_{\R^n} \frac{|f(x)-f(x-\eps z)|^{p}}{|\eps z|^{sp+n}}|\eps z|^{sp+n}\varphi(z)\, dz\, dx \\
& \leq \eps^{sp+n}\int_{\R^n} \int_{\R^n} \frac{|f(x)-f(x-\eps z)|^{p}}{|\eps z|^{sp+n}}\sup_{z}\{|z|^{sp+n}\varphi(z)\} dz\, dx \\
& \leq C \eps^{sp+n}\int_{\R^n} \int_{\R^n} \frac{|f(x)-f(x-\eps z)|^{p}}{|\eps z|^{sp+n}} dz\, dx = C \eps^{sp+n}\int_{\R^n} \int_{\R^n} \frac{|f(x)-f(y)|^{p}}{|x-y|^{sp+n}}\frac{1}{\eps^{n}}dy\, dx \\
& \leq C \eps^{sp}\|f\|_{W^{s,p}}^{p} \,,
\end{align*}
where in the forth line we used Jensen's inequality applied with the measure $\varphi \cdot \mathcal{L}^{n}$. This proves the first inequality in the statement. 

For the second estimate we compute
\begin{align*}
& \|Df^{\eps}\|^p_{L^{p}(\R^{n})}= \| f \ast D\varphi^{\eps}\|^p_{L^{p}(\R^{n})}=\int_{\R^n}\left|\int_{\R^n} \!\!f(x-y) D\varphi^{\eps}(y)\,dy\right|^{p}dx \\
& = \int_{\R^n}\left|\int_{\R^n}\!\!f(x-y) D_y\left(\frac{1}{\eps^{n}}\varphi\left(\frac{y}{\eps}\right)\right)\,dy\right|^{p} dx = \int_{\R^n}\left|\int_{\R^n}\!\!f(x-\eps z) D_z\varphi(z)\frac{1}{\eps^{n+1}} \eps^{n}\,dz \right|^{p} dx \\
&= \frac{1}{\eps^{p}}\int_{\R^n}\left|\int_{B_1}\!\!f(x-\eps z) D_z\varphi(z)\,dz\right|^{p} dx = \frac{1}{\eps^{p}}\int_{\R^n}\left|\int_{B_1}\!\!f(x-\eps z) D_z\varphi(z)\,dz -\int_{\R^n}\!\! f(x)D_z\varphi(z)\, dz \right|^{p} dx \\
& = \frac{1}{\eps^{p}}\mathcal{L}^{n}(B_1)^{p}\int_{\R^n}\left|\int_{B_1}\!\![f(x-\eps z)- f(x)]D_z\varphi(z)\, \frac{dz}{\mathcal{L}^{n}(B_1)} \right|^{p} dx              \\
& \leq \frac{1}{\eps^{p}}\mathcal{L}^{n}(B_1)^{p}\int_{\R^n}\int_{B_1} \left| [f(x-\eps z)-f(x)] D_z\varphi(z)\right|^{p}\frac{dz}{\mathcal{L}^{n}(B_1)} dx \\
& = \frac{1}{\eps^{p}}\mathcal{L}^{n}(B_1)^{p-1}\int_{\R^n}\int_{B_1} \left| [f(x-\eps z)-f(x)] D_z\varphi(z)\right|^{p}dz\, dx \\
& = \frac{1}{\eps^{p}}\mathcal{L}^{n}(B_1)^{p-1}\int_{\R^n}\int_{B_1} \frac{ |f(x-\eps z)-f(x)|^{p}}{|\eps z|^{sp+n}}|\eps z|^{sp+n}|D_z\varphi(z)|^{p}dz\, dx \\
& \leq \frac{1}{\eps^{p}}\eps^{sp+n}C_{n}\int_{\R^n}\int_{B_1} \frac{ |f(x-\eps z)-f(x)|^{p}}{|\eps z|^{sp+n}}\sup_z\{| z|^{sp+n}|D_z\varphi(z)|^{p}\}dz\, dx \\
& \leq C \eps^{p(s-1)}\eps^{n}\int_{\R^n}\int_{\R^n} \frac{ |f(x)-f(y)|^{p}}{|x-y|^{sp+n}}\frac{1}{\eps^{n}} dy\, dx \\
& \leq C \eps^{p(s-1)}\|f\|_{W^{s,p}}^{p} \,,
\end{align*}
where in the third line we used that $D_z \varphi$ has zero average, and in the fifth line we used Jensen's inequality for the measure$\frac{1}{\mathcal{L}^{n}(B_1)}\cdot\mathcal{L}^{n}$.
\end{proof}

\subsection{Maximal estimates}

In the course of the proof of our main theorem we will several times need to estimate difference quotients 
of the vector field. We will follow the strategy in~\cite{CDL} and rely on suitable maximal estimates. We now
briefly recall the main definitions, the most classical version of these estimates, and some anisotropic variants 
proved in~\cite{BBC}.

\begin{definition}\label{d:maxfct}
For any integrable function $u : \R^n \to \R$ the maximal function of $u$ is defined as
$$
Mu(x)= \sup_{r>0}\frac{1}{\mathcal{L}^{n}(B(x,r))}\int_{B(x,r)}|u(z)|\, dz \,,
\qquad x \in \R^n\,.
$$
\end{definition}

It can be shown that, for $u\in L^1(\R^n)$, the maximal function $Mu$ is a.e.~finite. Moreover, the following norm estimates hold (see~\cite{St1,St2} for a proof):

\begin{lemma}\label{l:maxest}
For any $1<p\leq \infty$ the strong estimate
$$
\| Mu\|_{L^p(\R^n)} \leq C \| u\|_{L^p(\R^n)}
$$
holds, where $C$ depends on $p$ and $n$ only. For $p=1$ only the weak etimate
$$
||| Mu|||_{M^1(\R^n)} \leq C \| u\|_{L^1(\R^n)}
$$
holds, with $C$ depending on $n$ only. In the above we denoted by
\begin{equation}\label{e:wL1}
||| f |||_{M^1(\R^n)} = \sup_{\lambda>0} \big\{ \lambda \, \mathcal{L}^n (\{ x : |f|>\lambda\}) \big\}
\end{equation}
the weak-$L^1$ norm.
\end{lemma}

The basic maximal estimate for the difference quotients of a Sobolev function is the following one. We recall its classical proof for the reader's convenience.

\begin{lemma}
\label{firstlemma}
Let $f:\R^{n}\rightarrow\R$ be a function in $W^{1,1}(\R^{n})$. Then for a.e. $x,y \in \R^{n}$,
$$
| f(x)-f(y)|\leq C_n |x-y|\big(  MDf(x)+MDf(y)\big) \,.
$$
\end{lemma}

\begin{proof}
First we prove the estimate for $f\in C^{1}$. We denote
\begin{equation}
\begin{split}
A=B\left( \frac{x+y}{2}, \frac{|x-y|}{2}\right) \qquad & A_{t,x}=tx+(1-t)A \qquad  B_{t,x}=B(x,(1-t)|x-y|) \\
& A_{t,y}=ty+(1-t)A \qquad  B_{t,y}=B(y,(1-t)|x-y|) \,.
\end{split}
\end{equation}
We note that $A_{t,x}\subset B_{t,x}$ and $A_{t,y}\subset B_{t,y}$. We estimate
\begin{align*}
& |f(x)-f(y)|=\intbar_{A}|f(x)-f(y)|\, dz \leq  \intbar_{A}|f(x)-f(z)|\, dz + \intbar_{A}|f(y)-f(z)|\, dz   \\
& = \intbar_{A} \left| \int_0^{1}\frac{d}{dt}[f(tx+(1-t)z] dt    \right| dz + \intbar_{A} \left| \int_0^{1}\frac{d}{dt}[f(ty+(1-t)z] dt    \right| dz \\
& \leq \frac{1}{\mathcal{L}^{n}(A)}\int_{A}  \int_0^{1}\left|\frac{d}{dt}[f(tx+(1-t)z]\right| dt  \,  dz +\frac{1}{\mathcal{L}^{n}(A)} \int_{A} \int_0^{1} \left|\frac{d}{dt}[f(ty+(1-t)z] \right| dt  \, dz \\
& \leq \frac{1}{\mathcal{L}^{n}(A)} \left[ \int_0^{1} \! \!\int_{A} |Df(tx+(1-t)z) |\, |x-z| dz dt +  \int_0^{1}\!\! \int_{A}|Df(ty+(1-t)z) |\, |y-z|  dz dt\right] \\
& \leq \frac{1}{\mathcal{L}^{n}(A)}|x-y| \left[ \int_0^{1} \! \!\int_{A} |Df(tx+(1-t)z) | dz  dt +  \int_0^{1}\!\! \int_{A}|Df(ty+(1-t)z) | dz dt\right].
\end{align*}
We apply a change of variable and we obtain that the last line equals
\begin{align*}
& \frac{1}{\mathcal{L}^{n}(A)}|x-y| \left[ \int_0^{1} \! \!\int_{A_{t,x}} |Df(w) | \frac{dw}{1-t} dt +  \int_0^{1}\!\! \int_{A_{t,y}}|Df(w) | \frac{dw}{1-t} dt\right]  \\
& \leq  \frac{1}{\mathcal{L}^{n}(A)}|x-y| \left[ \int_0^{1} \frac{\mathcal{L}^{n} ( B_{t,x})}{1-t} \frac{1}{\mathcal{L}^{n}(B_{t,x})} \int_{B_{t,x}}|Df(w) | dw\, dt \;+\; \text{symmetric} \right]  \\
& \leq \frac{n}{\frac{|x-y|^{n}}{2^{n}} (2\pi )^{\frac{n}{2}}}  |x-y| \left[ \int_0^{1}    \frac{  \frac{(1-t)^{n}}{n}|x-y|^{n}(2\pi)^{\frac{n}{2}}        }{1-t } \sup_{r>0}\intbar_{B(x,r)} |Df(w)|dw\, dt \;+\; \text{symmetric} \right] \\
& = 2^{n}|x-y| \int_0^{1}(1-t)^{n-1}\! dt\, \big[ MDf(x)+MDf(y) \big] \\
& = C_n|x-y|  \big[ MDf(x)+MDf(y) \big]\,,
\end{align*}
where we used $\mathcal{L}^{n}(B(x,r))= \frac{r^{n}(2\pi)^{\frac{n}{2}}}{n}$.

To conclude the proof for  $f\in W^{1,1}(\R^{n})$ it suffices to approximate $f$ with a sequence $(f_{\eps})\subset C^{1}(\R^{n})$ which converges to $f$ in $W^{1,1}(\R^{n})$ as $\eps \to 0$. \end{proof}

In our main result we will deal with a vector field with partial regularity. This assumption entails a splitting of the space as $\R^N = \R^{n_1} \times \R^{n_2}$ (with $N=n_1+n_2$). We will denote the variable $x \in \R^N$ by $x=(x_1,x_2)$, where $x_1\in\R^{n_1}$ and $x_2 \in \R^{n_2}$. Following~\cite{BBC}, for $\delta_1$, $\delta_2>0$ we consider the $N \times N$ diagonal matrix
\begin{equation}\label{e:matrix}
A = \begin{bmatrix} \delta_1 & & & &  \\ & \delta_1 & & & \\  & & \ddots & & \\ & & & \delta_2 & \\ & & & & \delta_2\end{bmatrix} \,,
\end{equation}
where $\delta_1$ appears at the first $n_1$ entries on the diagonal, and $\delta_2$ at the remaining $n_2$. In other words, we have
$$
A (x_1,x_2) = (\delta_1x_1,\delta_2x_2)\,, \qquad (x_1,x_2) \in \R^{n_1} \times \R^{n_2}\,.
$$

The next two lemmas have been proved in larger generality in~\cite{BBC}. We state them in our setting and give a simpler proof for the reader's convenience. 

\begin{lemma}
\label{LemmaUrho}
Let $f:\R^{N}\rightarrow\R$ be a function in $W^{1,1}(\R^{N})$. Let $A$ be the matrix defined in~\eqref{e:matrix}. Then there exists a nonnegative function ${U}$ such that for a.e. $x,y \in \R^{N}$,
\[
| f(x)-f(y)|\lesssim |A^{-1}[x-y]|\left({U}(x)+{U}(y) \right),
\]
with
\[
{U}(x)= M\left(\sum_{j=1}^{N}|\partial_j f(A\cdot )| A_{jj}  \right)(A^{-1}x).
\]

\end{lemma}

\begin{proof}
The result follows from Lemma~\ref{firstlemma} above. We denote $\tilde{f}(z)=f(Az)$. Then we know that, for a.e. $z$, $w$,
\begin{equation}
\label{eq:tilde_f}
|\tilde{f}(z)-\tilde{f}(w)|\leq C_N |z-w| \big( MD\tilde{f}(z)+MD\tilde{f}(w)\big),
\end{equation}
where in addition we notice
\begin{equation}
\label{eq:MDf_tilde}
MD\tilde{f}(z)\leq M\left(\sum_{j=1}^{N}|\partial_j \tilde{f}|\right)(z) = M\left(\sum_{j=1}^{N}(|\partial_j f(A\cdot)|A_{jj})\right)(z).
\end{equation}
Combining \eqref{eq:tilde_f} and \eqref{eq:MDf_tilde} we have, for a.e.~$z,w$,
\begin{equation}
|f(Az)-f(Aw)|\leq C_N |z-w| \left(   M\sum_{j=1}^{N}(|\partial_j f(A\cdot)|A_{jj})(z)   + M\sum_{j=1}^{N}(|\partial_j f(A\cdot)|A_{jj})(w)  \right).
\end{equation}
Now from the last inequality, taking $x$ and $y$ such that $z=A^{-1}x$ and $w=A^{-1}y$, we obtain the thesis.
\end{proof}

\begin{lemma}[Operator bounds]
\label{LemmaUbounds}
Let ${U}$ be defined as in Lemma \ref{LemmaUrho}. Then we have the estimates
\begin{equation}
\label{eq:Ubounds}
||| {U}|||_{M^1(\R^{N})}\leq C \left( \delta_1 \sum_{j=1}^{n_1} ||\partial_j f||_{L^{1}(\R^{N})}+\delta_2 \sum_{j=n_1+1}^{N} ||\partial_j f||_{L^{1}(\R^{N})}    \right)
\end{equation}
for $\partial_j f \in L^{1}$, and
\begin{equation}
\label{Ubound2}
|| {U}||_{L^{p}(\R^{N})}\leq C \left( \delta_1 \sum_{j=1}^{n_1} ||\partial_j f||_{L^{p}(\R^{N})}+\delta_2 \sum_{j=n_1+1}^{N} ||\partial_j f||_{L^{p}(\R^{N})}    \right)
\end{equation}
for $\partial_j f \in L^{p}$.
\end{lemma}

\begin{proof}
As in Lemma \ref{LemmaUrho} we consider $\tilde{f}(z)=f(Az)$. 
We exploit the estimates in Lemma~\ref{l:maxest} to the effect that
\begin{align*}
|||{U}|||_{M^1(\R^{N})} & = \left|\left|\left| M\sum_{j=1}^{N}(|\partial_j f(A\cdot)|A_{jj})(A^{-1}\cdot)\right|\right|\right|_{M^{1}(\R^{N})} =\left|\left|\left| M\sum_{j=1}^{N}|\partial_j \tilde{f}|(A^{-1}\cdot)\right|\right|\right|_{M^{1}(\R^{N})}  \\
& \leq C \left|\left| \sum_{j=1}^{N}|\partial_j \tilde{f}|(A^{-1}\cdot)\right|\right|_{L^{1}(\R^{N})} \leq C\sum_{j=1}^{N} \| (\partial_j \tilde{f})(A^{-1}\cdot) \|_{L^{1}(\R^{N})} \\
& = C \sum_{j=1}^{N} \| (\partial_j f(A\cdot) A_{jj} )(A^{-1}\cdot)\|_{L^{1}(\R^{N})} = C \sum_{j=1}^{N} A_{jj}\| \partial_j f \|_{L^{1}(\R^{N})} \,, 
\end{align*}
which is equation \eqref{eq:Ubounds}. With similar computations we can obtain \eqref{Ubound2}.
\end{proof}

We close this section with the following interpolation lemma, which allows to estimate the $L^1$ norm in terms of the weak-$L^1$ norm defined in~\eqref{e:wL1}, with a logarithmic dependence on higher integrability norms. 

\begin{lemma}[Interpolation]\label{l:inter}
Let $u: \Omega \to [0,+\infty)$ be a nonnegative measurable function, where $\Omega\subset\IR^n$ has finite measure. Then for every $1<p<\infty$, we have the interpolation estimate
$$
\|u\|_{L^1(\om)}\leq\frac{p}{p-1}|||u|||_{M^1(\om)}\left[1+\log\left(\displaystyle\frac{\|u\|_{L^p(\om)}}{|||u|||_{M^1(\om)}}\La^n(\om)^{1-\frac{1}{p}}\right)\right] \,,
$$
and analogously for $p=\infty$
$$
\|u\|_{L^1(\om)}\leq |||u|||_{M^1(\om)}\left[1+\log\left(\displaystyle\frac{\|u\|_{L^\infty(\om)}}{|||u|||_{M^1(\om)}}\La^n(\om) \right)\right] \,.
$$
\end{lemma} 


\section{Main result and corollaries}

\subsection{Assumptions on the vector field}
\label{ss:assumpt}

We recall that we consider a splitting of the space as $\R^N = \R^{n_1}\times\R^{n_2}$ and that we denote the variable by~$x=(x_1,x_2)$, with $x_1\in \R^{n_1}$ and $x_2 \in \R^{n_2}$. We are dealing with a vector field\break $b : (0,T)\times \R^{n_1}\times\R^{n_2} \to \R^{n_1}\times\R^{n_2}$ for which we assume the following regularity:

\begin{equation}
\label{hp_on_b}
\begin{split}
\mathbf{(R2)}: \quad &b(s,x_1,x_2)=\big(b_1(s,x_1),b_2(s,x_1,x_2)\big) \in \R^{n_1}\times \R^{n_2}=\R^{N} \\
&\quad b_1(s,x_1)\in L^1\big((0,T);W^{1,p}_{x_1}(\R^{n_1})\big) \\
&\quad b_2(s,x_1,x_2) \in L^{1}\big((0,T)\times\R^{n_2}_{x_2} ; W^{\alpha,1}_{x_1}(\R^{n_1})\big)
\cap L^{1} \big((0,T)\times\R^{n_1}_{x_1} ; W^{1,p}_{x_2}(\R^{n_2})\big)\,,
\end{split}
\end{equation}
for some given $p>1$ and $1/2 < \alpha <1$.

Moreover, we will assume that 
\begin{equation}
\label{hp_integ}
\begin{split}
\mathbf{(R3)}: \quad &b(t,x_1,x_2) \in L^p_{\rm loc}((0,T)\times\R^N)\,.
\end{split}
\end{equation}
Also recall that suitable growth conditions on $b$ have been assumed in {\bf (R1)}.

Let us introduce some further notation that will be used in the following.
 
We denote by $D_i{b}_j=D_{x_i}{b}_j$ the partial derivatives in distributional sense. We set $D_1{b}_1=p(t,x_1)$, $D_1{b}_2=q(t,x_1,x_2)$, and $D_2{b}_2 =r(t,x_1,x_2)$. Then we have
\begin{equation}\label{e:matreg}
\begin{split}
& Db=
\begin{pmatrix}
D_1b_1 & D_2b_1 \\
D_1b_2 & D_2 b_2
\end{pmatrix}
=
\begin{pmatrix}
p & 0 \\
q & r
\end{pmatrix}
\in
\begin{pmatrix}
L^{1}_{\text{$x_2$,loc}}L^{p}_{x_1} & 0 \\
\text{distribution} & L^{1}_{x_1}L^{p}_{x_2}
\end{pmatrix}\,.
\end{split}
\end{equation}

\subsection{Main estimate for the Lagrangian flow}

\begin{theorem} \label{theorem}
Let $b$ and $\bar{b}$ be two vector fields satisfying assumptions {\bf (R1)}. Assume the following:
\begin{itemize}
\item The second component of $\bar{b}$ satisfies $\bar{b}_2 \in L^{1}\big((0,T)\times\R^{n_2}_{x_2} ; W^{\alpha,1}_{x_1}(\R^{n_1})\big)$, 
\item The vector field $b$ satisfies {\bf (R2)} and {\bf (R3)}. 
\end{itemize}
Let $X$ and $\bar{X}$ be regular Lagrangian flows associated to $b$ and $\bar{b}$ respectively, with compressibility constants $L$ and $\bar{L}$. Then the following holds. For every positive $\gamma$, $r$ and $\eta$ there exists $\la >0$ and $C_{\gamma,r,\eta}>0$ such that
\begin{equation}\label{e:mainest}
\mathcal{L}^{N}\left(B_r \cap \left\lbrace |X(s,\cdot)-\bar{X}(s,\cdot)|>\gamma \right\rbrace \right) \leq C_{\gamma,r,\eta} \| b-\bar{b} \|_{L^{1}((0,T)\times B_{\la})}+ \eta
\end{equation}
for all $s \in [0,T]$, where $C_{\gamma,r,\eta}$ depends on $L$, $\bar{L}$, the bound for $\bar{b}_2$ in $L^{1}\big((0,T)\times\R^{n_2}_{x_2} ; W^{\alpha,1}_{x_1}(\R^{n_1})\big)$, the bound for the decomposition of $\bar{b}$ as in {\bf (R1)}, 
and the various bounds for $b$ involved in the assumptions~{\bf (R1)}, {\bf (R2)}, and {\bf (R3)}.
\end{theorem}

\begin{proof} We exploit the anisotropic functional
\begin{equation}\label{e:functional}
\Phi_{\delta_1,\delta_2}(s)= \int_{B_r\cap G_{\la}\cap \bar{G}_{\la}} \log\left( 1+ | A^{-1}[X(s,x)-\bar{X}(s,x)] |  \right) \, dx\,,
\end{equation}
where the matrix $A$ has been defined in~\eqref{e:matrix} and $G_{\la}$ (respectively, $\bar{G}_{\la}$) are the sublevels of the regular Lagrangian flow $X$ (respectively, $\bar{X}$) defined as in~\eqref{e:sublevels}.

\medskip

\noindent\textbf{Step 1:} \textit{Regularization of the vector field.} 
We regularize ${b}_2$ by convolution in $x_1$. Let $\varphi^{\eps}$ be a standard mollifier in $\R^{n_1}$. We denote the regularization of ${b}_2$ by
\[   
{b}_2^{\eps}(t,x_1,x_2)={b}_2 (t,x_1,x_2) \ast_{x_1} \varphi^{\eps}(x_1)\,,
\qquad\quad\text{for $t$ and $x_2$ fixed,}
\]
and we further denote $b^{\eps}= (b_1,{b}_2^{\eps})$. Moreover, $q^{\eps}$ and $r^{\eps}$ are associated to $b^\eps$ as in~\eqref{e:matreg}.

Due to standard properties of the convolution we have that
$b^{\eps}\rightarrow b$ and $r^{\eps}\rightarrow r$ in $L^{1}_{\text{loc}}(\R^{N})$.
Also recall the rates of convergence and blow-up proved in Lemma~\ref{lemma_W^{s,p}}.

\medskip

\noindent\textbf{Step 2:} \textit{Time differentiation.} 
By differentiating the functional $\Phi_{\delta_1,\delta_2}(s)$ with respect to time we get
\begin{align*}
& \Phi'_{\delta_1,\delta_2}(s) \leq  \int_{B_r\cap G_{\la}\cap \bar{G}_{\la}} \frac{|A^{-1}[b(s,X)-\bar{b}(s,\bar{X})]|}{1+|A^{-1}[X-\bar{X}]|}  dx\\
& \leq \int_{B_r\cap G_{\la}\cap \bar{G}_{\la}} \frac{|A^{-1}[b(s,X)-b^{\eps}(s,X)]|}{1+|A^{-1}[X-\bar{X}]|} + \frac{|A^{-1}[\bar{b}^{\eps}(s,\bar{X})-\bar{b}(s,\bar{X})]|}{1+|A^{-1}[X-\bar{X}]|} +\frac{|A^{-1}[b^{\eps}(s,X)-\bar{b}^{\eps}(s,\bar{X})]|}{1+|A^{-1}[X-\bar{X}]|}  dx \\
 &\leq \frac{L}{\delta_1} \| b-b^{\eps}(s,\cdot)\|_{L^{1}(B_{\la})} + \frac{\bar{L}}{\delta_1} \| \bar{b}-\bar{b}^{\eps}(s,\cdot)\|_{L^{1}(B_{\la})}   \\
&\;\;+ \int_{B_r\cap G_{\la}\cap \bar{G}_{\la}} \frac{|A^{-1}[b^{\eps}(s,X) -b^{\eps}(s,\bar{X})+ b^{\eps}(s,\bar{X})   -\bar{b}^{\eps}(s,\bar{X})]|}{1+|A^{-1}[X-\bar{X}]|} \, dx\\
&\leq \frac{L}{\delta_1} \| b-b^{\eps}(s,\cdot)\|_{L^{1}(B_{\la})} + \frac{\bar{L}}{\delta_1} \| \bar{b}-\bar{b}^{\eps}(s,\cdot)\|_{L^{1}(B_{\la})} + \frac{\bar{L}}{\delta_1} \| b^{\eps}-\bar{b}^{\eps}(s,\cdot)\|_{L^{1}(B_{\la})}   \\
 &\;\;+\int_{B_r\cap G_{\la}\cap \bar{G}_{\la}} \frac{|A^{-1}[b^{\eps}(s,X) -b^{\eps}(s,\bar{X})]|}{1+|A^{-1}[X-\bar{X}]|} \, dx\\
 & \leq \frac{L}{\delta_1} \| b-b^{\eps}(s,\cdot)\|_{L^{1}(B_{\la})} + \frac{\bar{L}}{\delta_1} \| \bar{b}-\bar{b}^{\eps}(s,\cdot)\|_{L^{1}(B_{\la})} + \frac{\bar{L}}{\delta_1} \| b^{\eps}-\bar{b}^{\eps}(s,\cdot)\|_{L^{1}(B_{\la})}   \\
 & \;\; + \int_{B_r\cap G_{\la}\cap \bar{G}_{\la}} \min\left\lbrace |A^{-1}[b^{\eps}(s,X) -b^{\eps}(s,\bar{X})]|, \frac{|A^{-1}[b^{\eps}(s,X) -b^{\eps}(s,\bar{X})]|}{|A^{-1}[X-\bar{X}]|}       \right\rbrace dx.
\end{align*}

\medskip

\noindent\textbf{Step 3: }\textit{Bounds with maximal operators.} Integrating in time 
and recalling the definition of the matrix~$A$ in~\ref{e:matrix} we get
\begin{equation}
\label{eq:Step2}
\begin{split}
& \Phi_{\delta_1,\delta_2}(\tau) \leq  \frac{L}{\delta_1} \| b-b^{\eps}\|_{L^{1}((0,\tau)\times B_{\la})} + \frac{\bar{L}}{\delta_1} \| \bar{b}-\bar{b}^{\eps}\|_{L^{1}((0,\tau)\times B_{\la})} + \frac{\bar{L}}{\delta_1} \| b^{\eps}-\bar{b}^{\eps}\|_{L^{1}((0,\tau)\times B_{\la})}   \\
 & +\!\! \int_0^{\tau}\!\!\int_{B_r\cap G_{\la}\cap \bar{G}_{\la}} \!\!\min\Big\lbrace |A^{-1}[b^{\eps}(s,X) -b^{\eps}(s,\bar{X})]|, \frac{1}{\delta_1}\frac{|b_1(s,X) -b_1(s,\bar{X})|}{|A^{-1}[X-\bar{X}]|} \\
& \hspace{9cm} \!+\!\frac{1}{\delta_2}\frac{|b_2^{\eps}(s,X) -b_2^{\eps}(s,\bar{X})|}{|A^{-1}[X-\bar{X}]|}       \Big\rbrace dx  ds.
\end{split}
\end{equation}

Lemmas~\ref{LemmaUrho} and~\ref{LemmaUbounds} can be easily extended to vector valued functions. We would like to apply these lemmas to $b^{\eps}$, which is only locally $W^{1,1}$ in $\R^N$, as the first component $b_1$ does not depend on~$x_2$. This can be done by defining a new vector field $\tilde{b}^{\eps}$ as the smooth cut-off of $b^{\eps}$ on the ball of radius $2\la$, i.e. $\tilde{b}^{\eps}=b^{\eps}\cdot \chi_{B_\la}=(b_1\cdot \chi_{B_\la},b^{\eps}_2\cdot \chi_{B_\la})=(\tilde{b}_1,\tilde{b}^{\eps}_2)$, where $\chi_{B_\la}$ is a smooth function with value~$1$ on $B_{2\la}$ and $0$ on $\R^{N}\setminus B_{2\la + 1}$, and by using suitable truncated maximal functions in the maximal estimates. We define $\tilde{p}$, $\tilde{q}$, $\tilde{r}$, $\tilde{q}^{\eps}$ and $\tilde{r}^{\eps}$ as the partial derivatives of $\tilde{b}$ $(=b\cdot \chi_{B_{\la}})$ and $\tilde{b}^{\eps}$. 

Lemma \ref{LemmaUrho} applied to $\tilde{b}_1$ and $\tilde{b}^{\eps}_2$ yields 
\begin{equation}
\frac{|\tilde{b}_1(s,x)-\tilde{b}_1(s,\bar{x})|}{|A^{-1}[x-\bar{x}]|}\lesssim U_{\tilde{p}}(x)+  U_{\tilde{p}}(\bar{x}),\,\, 
 \end{equation}
and
 \begin{equation}
 \frac{|\tilde{b}_2^{\eps}(s,x)-\tilde{b}_2^{\eps}(s,\bar{x})|}{|A^{-1}[x-\bar{x}]|}\lesssim {U}_{\tilde{q}^{\eps},\tilde{r}^{\eps}}(x)+  {U}_{\tilde{q}^{\eps},\tilde{r}^{\eps}}(\bar{x})
\end{equation} 
for $s\in [0,T]$, and for a.e.~$x, \bar{x} \in \R^{N}$.

By subadditivity of ${U}$ we can estimate

\[
{U}_{\tilde{q}^{\eps},\tilde{r}^{\eps}} \leq {U}_{\tilde{q}^{\eps}}+{U}_{\tilde{r}^{\eps}},
\]
implying that 
\begin{align*}
&\frac{|\tilde{b}_1(s,x)-\tilde{b}_1(s,\bar{x})|}{|A^{-1}[x-\bar{x}]|}\lesssim U_{\tilde{p}}(x) + U_{\tilde{p}}(\bar{x}),\,\, \\
 &\frac{|\tilde{b}_2^{\eps}(s,x)-\tilde{b}_2^{\eps}(s,\bar{x})|}{|A^{-1}[x-\bar{x}]|}\lesssim {U}_{\tilde{q}^{\eps}}(x)+U_{\tilde{r}^{\eps}}(x)+{U}_{\tilde{q}^{\eps}}(\bar{x})+{U}_{\tilde{r}^{\eps}}(\bar{x}).
\end{align*}

\smallskip
\noindent \textbf{Step 4: } \textit{Estimates for the maximal operators.} 
Let $\Omega=(0,\tau)\times \big( B_r \cap G_{\la}\cap \bar{G}_{\la} \big) \subset \R^{N+1}$. We can estimate the last term of the sum \eqref{eq:Step2} with
\begin{align*}
&  \int_{\Omega} \min\Big\lbrace |A^{-1}[b^{\eps}(s,X)-b^{\eps}(s,\bar{X})]|,\frac{1}{\delta_1} \left(U_{\tilde{p}}(s,X)+U_{\tilde{p}}(s,\bar{X})\right) \\
& \qquad\qquad\qquad\qquad\qquad + \frac{1}{\delta_2}\left( (U_{\tilde{q}^{\eps}}+U_{\tilde{r}^{\eps}}) (s,X)+ (U_{\tilde{q}^{\eps}}+U_{\tilde{r}^{\eps}})(s,\bar{X}) \right)  \Big\rbrace dx\, ds =: \tilde{\Phi}_{\delta_1,\delta_2}(\tau) \,.
\end{align*}

Lemma~\ref{LemmaUbounds} implies
\begin{align*}
||| U_{\tilde{q}^{\eps}}|||_{M^{1}((0,T)\times B_{\la})} \lesssim \delta_1 \| \tilde{q}^{\eps}\|_{L^{1}((0,T)\times \R^{N})}=\delta_1 \| \tilde{q}^{\eps}\|_{L^{1}((0,T)\times B_{2\la+1})}\leq \delta_1 \| {q}^{\eps}\|_{L^{1}((0,T)\times B_{2\la+1})} =: \delta_1 \psi(\eps)\,.
\end{align*}
Notice that the quantity $\psi(\eps)$ at the right hand side could a priori blow up as $\eps \rightarrow 0$, as we are not assuming that $q=D_1b_2$ is integrable.

Splitting the minima once again, we obtain
\begin{align*}
& \tilde{\Phi}_{\delta_1,\delta_2}(\tau) \leq \int_{\Omega} \min\left\lbrace |A^{-1}[b^{\eps}(s,X)-b^{\eps}(s,\bar{X})]|, \frac{1}{\delta_2}\left( U_{\tilde{q}^{\eps}}(s,X)+ U_{\tilde{q}^{\eps}}(s,\bar{X}) \right)  \right\rbrace dx\, ds \\
& \quad + \int_{\Omega} \min\left\lbrace |A^{-1}[b^{\eps}(s,X)-b^{\eps}(s,\bar{X})]|, \frac{1}{\delta_2}\left( U_{\tilde{r}^{\eps}}(s,X)+ U_{\tilde{r}^{\eps}}(s,\bar{X}) \right)  \right\rbrace dx\, ds \\
& \quad + \int_{\Omega} \min\left\lbrace |A^{-1}[b^{\eps}(s,X)-b^{\eps}(s,\bar{X})]|, \frac{1}{\delta_1}\left( U_{\tilde{p}}(s,X)+ U_{\tilde{p}}(s,\bar{X})\right)  \right\rbrace dx\, ds \\
& = \int_{\Omega}  \varphi_1(s,X,\bar{X}) \, dxds+ \int_{\Omega}  \varphi_2(s,X,\bar{X}) \, dxds +\int_{\Omega}  \varphi_3(s,X,\bar{X}) \, dxds \,.
\end{align*}

Let $\Omega'=(0,\tau)\times B_{\la} \subset \R^{N+1}$. Using the first element of the minimum and relying on assumption {\bf (R3)} we can estimate
\begin{equation*}
\| \varphi_1 \|_{L^{p}(\Omega)} \leq \frac{L^{1/p}+\bar{L}^{1/p}}{\delta_1} \|b^{\eps}\|_{L^{p}(\Omega')} \lesssim \frac{1}{\delta_1} \|b^{\eps}\|_{L^{p}(\Omega')}\lesssim \frac{1}{\delta_1} \|b\|_{L^{p}(\Omega')} \simeq \frac{1}{\delta_1}.
\end{equation*}

Exploiting the second term of the minimum, we get
$$
 |||\varphi_1|||_{M^{1}(\Omega)} \leq \frac{1}{\delta_2} ||| U_{\tilde{q}^{\eps}}(X)+U_{\tilde{q}^{\eps}}(\bar{X}) |||_{M^{1}(\Omega)} \lesssim \frac{1}{\delta_2} |||U_{\tilde{q}^{\eps}}|||_{M^{1}(\Omega' )}\lesssim \frac{\delta_1}{\delta_2}\|q_{\eps}\|_{L^{1}((0,T)\times B_{2\la+1})}= \frac{\delta_1}{\delta_2}\psi(\eps) \,.
$$
For $\varphi_2$ and $\varphi_3$ using assumption~{\bf (R2)} we have
\begin{equation}\label{L11}
\|\varphi_2\|_{L^{1}(\Omega)}\lesssim \frac{1}{\delta_2}\| U_{\tilde{r}^{\eps}}\|_{L^{1}(\Omega')} \lesssim_{\la}\frac{1}{\delta_2}\| U_{\tilde{r}^{\eps}}\|_{L^1((0,T);L^{p}(B_\la))} \lesssim \frac{\delta_2}{\delta_2} \|\tilde{r}_2^{\eps}\|_{L^{1}((0,T);L^p(\R^{N}))} \lesssim  C
\end{equation}
and 
\begin{equation}\label{L12}
\|\varphi_3\|_{L^{1}(\Omega)} \lesssim \frac{1}{\delta_1}\|U_{\tilde{p}}\|_{L^{1}(\Omega')} \lesssim_{\la}\frac{1}{\delta_1}\|U_{\tilde{p}}\|_{L^1((0,T);L^{p}(B_\la))}\lesssim \frac{\delta_1}{\delta_1}\|\tilde{p}_2\|_{L^{1}((0,T);L^p(\R^{N}))} \lesssim C\,.
\end{equation}

\medskip

\noindent\textbf{Step 5:} \textit{Interpolation Lemma.} We can apply now Lemma~\ref{l:inter} to $\varphi_1$, to the effect that
\begin{align*}
\Phi_{\delta_1,\delta_2}(\tau) & \lesssim_{\la}  \frac{1}{\delta_1} \| b^{\eps}-\bar{b}^{\eps}\|_{L^{1}(\Omega')} + \frac{1}{\delta_1} \sigma(\eps) + \frac{1}{\delta_1} \bar{\sigma}(\eps)+ \frac{\delta_1}{\delta_2}\psi(\eps)\log \left( \frac{1}{\frac{\delta_1}{\delta_2}\psi(\eps)\delta_1} \right)+ C  \\
& \lesssim \frac{1}{\delta_1}\| b-\bar{b}\|_{L^{1}(\Omega')}+ \frac{1}{\delta_1} \big[\sigma(\eps)+\bar{\sigma}(\eps)\big]+ \frac{\delta_1}{\delta_2}\psi(\eps)\log \left( \frac{1}{\frac{\delta_1}{\delta_2}\psi(\eps)\delta_1} \right) + C\,,
\end{align*}
where $\sigma(\eps) = \| b-b^{\eps}\|_{L^{1}(\Omega')}$ and $\bar{\sigma}(\eps)=\| \bar{b}-\bar{b}^{\eps}\|_{L^{1}(\Omega')}$ tend to $0$ as $\eps \rightarrow 0$. Lemma~\ref{lemma_W^{s,p}} implies that
\begin{equation}
\label{equinormWs,1}
\sigma(\eps) + \bar{\sigma}(\eps) \lesssim \left(\|b_2\|_{L^{1}_{t,x_2}W^{\alpha,1}_{x_1}}+\|\bar{b}_2\|_{L^{1}_{t,x_2}W^{\alpha,1}_{x_1}}\right)\eps^{\alpha} \quad \text{and} \quad \psi(\eps) \lesssim \left(\|b_2\|_{L^{1}_{t,x_2}W^{\alpha,1}_{x_1}} \right) \eps^{\alpha-1}. 
\end{equation}
Therefore 
\begin{align*}
& \mathcal{L}^{N}\left(B_r \cap \left\lbrace |X(s,\cdot)-\bar{X}(s,\cdot)|>\gamma \right\rbrace \right) \\
& \lesssim_{\la} \frac{\| b-\bar{b} \|_{L^{1}(\Omega')}}{\delta_1 \log(1+\frac{\gamma}{\delta_2})}+ \frac{\sigma(\eps)+\bar{\sigma}(\eps)}{\delta_1\log(1+\frac{\gamma}{\delta_2})} + \frac{\frac{\delta_1}{\delta_2}\psi(\eps)\log \left( \frac{1}{\frac{\delta_1}{\delta_2}\psi(\eps)\delta_1} \right)}{\log(1+\frac{\gamma}{\delta_2})}+ \frac{C}{\log(1+\frac{\gamma}{\delta_2})} \\
& \hspace{10cm}  + \mathcal{L}^{N}(B_r\setminus G_{\la})+ \mathcal{L}^{N}(B_r\setminus \bar{G}_{\la})   \\
& \lesssim \frac{\| b-\bar{b} \|_{L^{1}((0,T)\times B_{\la})}}{\delta_1 \log(1+\frac{\gamma}{\delta_2})}+ \frac{\eps^{\alpha}}{\delta_1\log(1+\frac{\gamma}{\delta_2})} + \frac{\frac{\delta_1}{\delta_2}\eps^{\alpha-1}\log \left( \frac{1}{\frac{\delta_1}{\delta_2}\eps^{\alpha-1}\delta_1} \right)}{\log(1+\frac{\gamma}{\delta_2})}+\frac{C}{\log(1+\frac{\gamma}{\delta_2})}  \\
& \hspace{10cm} + \mathcal{L}^{N}(B_r\setminus G_{\la})+ \mathcal{L}^{N}(B_r\setminus \bar{G}_{\la})  \\
& = \frac{\| b-\bar{b} \|_{L^{1}((0,T)\times B_{\la})}}{\delta_1 \log(1+\frac{\gamma}{\delta_2})} + 1)+2)+3)+4)+5) \,. 
\end{align*}

\medskip

\noindent\textbf{Step 6:} \textit{Choice of the parameters and conclusion.} Fix $\eta>0$. By choosing $\lambda$ sufficiently large we can make $4) + 5) \leq 2\eta/5$.

Define
$$
\beta = \frac{\delta_1}{\delta_2} \ll 1\,, \qquad \text{ so that $\delta_1 = \beta \delta_2$.}
$$
We need to choose $\eps>0$, $\beta>0$, and $\delta_2>0$ in such a way that
$$
1) + 2) + 3) =  \frac{\eps^{\alpha}}{\beta\delta_2\log(1+\frac{\gamma}{\delta_2})} + \frac{\beta\eps^{\alpha-1}\log \left( \frac{1}{\beta^2\eps^{\alpha-1}\delta_2} \right)}{\log(1+\frac{\gamma}{\delta_2})}+\frac{C}{\log(1+\frac{\gamma}{\delta_2})}  \leq \frac{3\eta}{5}\,.
$$
The term $3)$ can be made smaller than $\eta/5$ by choosing $\delta_2>0$ sufficiently small. We fix $0<\mu<1$ to be determined later (depending on the exponent $\alpha>1/2$ in assumption {\bf (R2)} only) and choose $\eps>0$ such that
$$
\eps^{\alpha-1}=\beta^{\mu-1}\,, \qquad \text{ that is, } \eps=\beta^{\frac{1-\mu}{1-\alpha}}\,.
$$
In this way we get
$$
2) = 
\frac{\beta^\mu\log \left( \frac{1}{\beta^{\mu+1}\delta_2} \right)}{\log(1+\frac{\gamma}{\delta_2})} 
= \frac{\beta^\mu\log \left( \frac{1}{\beta^{\mu+1}} \right)}{\log(1+\frac{\gamma}{\delta_2})}
+ \frac{\beta^\mu\log \left( \frac{1}{\delta_2} \right)}{\log(1+\frac{\gamma}{\delta_2})} \,,
$$
which can be made smaller that $\eta/5$ if $\beta>0$ is chosen to be small enough.

With the above choices the term $1)$ becomes
$$
1) = \frac{\beta^{\frac{1-\mu}{1-\alpha}\, \alpha}}{\beta\delta_2\log(1+\frac{\gamma}{\delta_2})}
= \frac{\beta^{\frac{2\alpha-\alpha\mu-1}{1-\alpha}}}{\delta_2\log(1+\frac{\gamma}{\delta_2})}\,,
$$
which can be made smaller than $\eta/5$ by a suitable choice of $\beta>0$, provided the exponent of $\beta$ at the numerator is positive, that is,
\begin{equation}\label{e:choicemu}
\frac{2\alpha-\alpha\mu-1}{1-\alpha} > 0
\qquad \Longleftrightarrow \qquad
\alpha> \frac{1}{2-\mu}\,.
\end{equation}
Since $\alpha>1/2$, we see that we can choose $\mu>0$ small enough in such a way that~\eqref{e:choicemu} holds. This gives $1)+2)+3)+4)+5)\leq \eta$ and therefore concludes the proof. \end{proof}

\subsection{Well-posedness and further properties of the Lagrangian flow} 
\label{ss:coroll}

Estimate~\eqref{e:mainest} in Theorem~\ref{theorem} is the key information which guarantees existence, uniqueness, and stability of the regular Lagrangian flow. The proof of these results as a consequence of estimate~\eqref{e:mainest} is by now quite standard, see the theory developed in~\cite{CDL,BC,BBC}. We begin with the uniqueness.

\begin{corollary}[Uniqueness]
Let $b$ be a vector field satisfying assumptions {\bf (R1)}, {\bf (R2)}, and {\bf (R3)}. Then, the regular Lagrangian flow associated to $b$, if it exists, is unique.
\end{corollary}

It is indeed very easy to see that uniqueness follows from estimate~\eqref{e:mainest}. We consider $b = \bar b$, then the right hand side of~\eqref{e:mainest} can be made arbitrarily small, for any $\gamma>0$ fixed. This readily implies uniqueness. 

\begin{remark} 
We observe that, in contrast to the PDE theory in~\cite{LBL,Ler1,Ler2}, no assumptions on the divergence of the vector field are required for the uniqueness of the regular Lagrangian flow. The divergence will play a role for the existence only. 
\end{remark} 

The main advantage of the quantitative theory of ODEs, in contrast to the PDE theory, is that it provides an explicit rate for the compactness and the stability, depending on the uniform bounds that are assumed on the sequence of vector fields. The following two results can be proven arguing as in~\cite{BC}, as a consequence of the main estimate~\eqref{e:mainest}. 

\begin{corollary}[Stability]
Let $\{b_n\}$ be a sequence of vector fields satisfying assumption {\bf (R1)}, converging in $L^{1}_{\loc}([0,T]\times \R^{N})$ to a vector field $b$ which satisfies assumptions {\bf (R1)}, {\bf (R2)}, and {\bf (R3)}. Assume that there exist $X_n$ and $X$ regular Lagrangian flows associated to $b_n$ and $b$ respectively, and denote by $L_n$ and $L$ the compressibility constants of the flows. Suppose that: 
\begin{itemize}
\item For some decomposition $b_n/(1+|x|)={c}_{n,1}+{c}_{n,2}$ as in assumption {\bf (R1)}, we have that 
$$
\| {c}_{n,1}\|_{L^{1}_t(L^{1}_x)}+ \|{c}_{n,2}\|_{L^{1}_t(L^{\infty}_x)} \quad \textit{is equi-bounded;}
$$
\item The sequence $\{L_n\}$ is equi-bounded;
\item The norm of $b_{n,2}(s,x_1,x_2)$ in $L^{1}\big((0,T)\times\R^{n_2}_{x_2} ; W^{\alpha,1}_{x_1}(\R^{n_1})\big)$ is equi-bounded.
\end{itemize}
Then the sequence $\{X_n\}$ converges to $X$ locally in measure in $\R^{N}$, uniformly with respect to time.
\end{corollary}

In the above corollary, the assumption in the third bullet is necessary in order to have a uniform estimate on the quantity $\sigma_n(\eps)$ associated to $b_n$ (as in the proof of Theorem~\ref{theorem}). 

\begin{corollary}[Compactness]\label{compactness}
Let $\{b_n\}$ be a sequence of vector fields satisfying assumption {\bf (R1)}, {\bf (R2)}, and {\bf (R3)}, converging in $L^{1}_{\loc}([0,T]\times \R^{N})$ to a vector field $b$ which satisfies assumptions {\bf (R1)}, {\bf (R2)}, and {\bf (R3)}. Assume that there exist $X_n$ regular Lagrangian flows associated to $b_n$, and denote by $L_n$ the compressibility constants of the flows. Suppose that: 
\begin{itemize}
\item For some decomposition $b_n/(1+|x|)={c}_{n,1}+{c}_{n,2}$ as in assumption {\bf (R1)}, we have that 
$$
\| {c}_{n,1}\|_{L^{1}_t(L^{1}_x)}+ \|{c}_{n,2}\|_{L^{1}_t(L^{\infty}_x)} \quad \textit{is equi-bounded;}
$$
\item The sequence $\{L_n\}$ is equi-bounded;
\item The norms of the vector fields $\{b_n\}$ involved in the assumptions {\bf (R2)} and {\bf (R3)} are equi-bounded. 
\end{itemize}
Then the sequence $\{X_n\}$ is pre-compact locally in measure in $\R^{N}$, uniformly with respect to time, and converges to a regular Lagrangian flow $X$ associated to $b$.
\end{corollary}

By a simple regularization procedure Corollary~\ref{compactness} implies existence of the regular Lagrangian flow, under the assumption of boundedness of the divergence of the vector field. Such an assumption is needed in order to have equi-boundedness of the compressibility constants for the sequence of approximated regular Lagrangian flows
$X_n$ in Corollary~\ref{compactness}. 

\begin{corollary}[Existence]
Let $b$ be a vector field satisfying assumptions {\bf (R1)}, {\bf (R2)}, and {\bf (R3)}. Assume that the (distributional) spatial divergence of $b$ is bounded. Then, there exists a regular Lagrangian flow associated to $b$.
\end{corollary}

\begin{remark} 
Arguing as in~\cite{BC}, it is also possible to develop a theory of Lagrangian solutions of the continuity equations, that is, solutions that are transported by the regular Lagrangian flow.
\end{remark}

\subsection{Remarks and possible extensions} 

We conclude by listing a few remarks and questions concerning the results and the approach in this work:
\begin{itemize}
\item[(1)] The same proof for Theorem~\ref{theorem} works if we assume only local regularity bounds in assumption~{\bf (R2)}. We omitted this just for simplicity of notation.  
\item[(2)] Compared to the PDE theory in~\cite{LBL,Ler1,Ler2}, we need to assume some fractional Sobolev regularity of $b_2$ with respect to the variable $x_1$. This seems unavoidable for our strategy of proof, since we cannot send to zero the two parameters $\delta_1$ and $\delta_2$ one after the other, but we rather need to send them together to zero, under a condition on their ratio $\beta = \delta_1/\delta_2$. Is it possible to modify our proof and remove this assumption, that is, is it possible to derive an estimate like~\eqref{e:mainest} under the only assumption of integrable depencence of $b_2$ with respect to $x_1$?
\item[(3)] Is it possible to treat the case $p=1$ in assumption~{\bf (R2)}? We briefly explain here what is the obstruction with the present approach. In the case $p=1$, in Step~4 of the main proof the operators $U_{\tilde{r}^{\eps}}$ and $U_{\tilde{p}}$ cannot be directly estimated in $L^1$ as in~\eqref{L11} and~\eqref{L12} (recall Lemma~\ref{LemmaUbounds}). One needs to argue as done in the same step for $U_{\tilde{q}^{\eps}}$ exploiting the equi-integrability and the interpolation from Lemma~\ref{l:inter}. After some computations we would obtain that, for every $\theta>0$, there is a constant $C_\theta>0$ so that the term
$$
\frac{C}{\log(1+\frac{\gamma}{\delta_2})} 
$$
in the last estimate at the end of Step~5 is replaced by the sum
$$
\frac{\theta \log\left(\frac{1}{\theta \beta\delta_2}\right)}{\log(1+\frac{\gamma}{\delta_2})}  
+\frac{C_{\theta}}{\log(1+\frac{\gamma}{\delta_2})} \,.
$$
We need to make also this sum small, exploiting the arbitrariness of $\theta$. We see that, in order to make the first term small, we need to take $\theta$ coupled to $\beta$. 
Choosing $\eps^{\alpha-1}=\beta^{\mu-1}$ as in the proof of Theorem~\ref{theorem}, we see that we still have $\beta$ and $\delta_2$ as free parameters, and eventually we need to make small the sum
$$
\frac{\beta^{\frac{2\alpha-\alpha\mu-1}{1-\alpha}}}{\delta_2\log(1+\frac{\gamma}{\delta_2})}
+
\frac{C_{\beta}}{\log(1+\frac{\gamma}{\delta_2})}
$$
(as now $\theta$ is coupled to $\beta$). However, since $C_\beta$ blows up for $\beta\to 0$ (depending on the equi-integrability rate), with this strategy there is in general no choice of such parameters which makes the last sum small.
\item[(4)] Can one relax the strong requirement that $b_1$ does not depend on the variable $x_2$, and require instead (for instance) that $b_1$ has a smooth dependence on $x_2$?
\end{itemize}


\bibliographystyle{plain}

\adresse

 \end{document}